\documentclass[12pt,oneside]{amsart}

\usepackage{hyperref}
\hypersetup{
    colorlinks=true,
    linkcolor=black,
    citecolor=magenta,
    filecolor=black,      
    urlcolor=blue,
    }
\usepackage{tkz-euclide,fullpage}
\tikzstyle{directed}=[postaction={decorate,decoration={markings,
    mark=at position .65 with {\arrow{stealth}}}}]
\tikzstyle{reverse directed}=[postaction={decorate,decoration={markings,
    mark=at position .5 with {\arrowreversed{stealth};}}}]
\usetikzlibrary{hobby}
\usetikzlibrary{positioning,fit,calc}
\usepackage{amssymb}
\usepackage{latexsym}
\usepackage{amsmath}
\usepackage{mathrsfs}
\usepackage[all]{xy}
\usepackage{enumerate}
\usepackage{amsthm}
\usepackage{psfrag}
\usepackage{ifthen}
\usepackage{mathdots}
\usepackage{caption}
\usepackage{subcaption}
\usepackage{standalone}
\usepackage{mathtools}
\usepackage{pgfplots}
\usepackage{xfrac}
\pgfplotsset{compat=1.11}
\usepackage[mathscr]{euscript}
\usepackage{bm}
\usepackage{amsbsy}
\usepackage[outline]{contour}
\usepackage{todonotes}
\usetikzlibrary{calc}

\newcommand{\newbold}[1]{\bgroup\contourlength{0.01em}\contour{black}{#1}\egroup}

\graphicspath{{graphics/}}

\newtheorem{theorem}{Theorem}[section]
\newtheorem{lemma}[theorem]{Lemma}
\newtheorem{metalemma}[theorem]{Meta Lemma}
\newtheorem{proposition}[theorem]{Proposition}
\newtheorem{corollary}[theorem]{Corollary}

\theoremstyle{definition}
\newtheorem{definition}[theorem]{Definition}

\newtheorem{example}[theorem]{Example}

\newenvironment{cor}{\begin{corollary}}{
\end{corollary}}

\usepackage{tabularx,environ}

\makeatletter
\newcommand{\problemtitle}[1]{\gdef\@problemtitle{#1}}
\newcommand{\probleminput}[1]{\gdef\@probleminput{#1}}
\newcommand{\problemquestion}[1]{\gdef\@problemquestion{#1}}
\NewEnviron{algproblem}{
  \problemtitle{}\probleminput{}\problemquestion{}
  \BODY
  \par\addvspace{.5\baselineskip}
  \noindent
  \begin{tabularx}{\textwidth}{@{\hspace{0em}} l X c}
    \hline\hline
    \multicolumn{2}{@{\hspace{0em}}l}{\@problemtitle} \\
    \textbf{In:} & \@probleminput \\%
    \textbf{Question:} & \@problemquestion\\
    \hline\hline
  \end{tabularx}
  \par\addvspace{.5\baselineskip}
}
\makeatother

\newcommand{\comment}[1]
	   {\ifthenelse{\equal{\showcomments}{yes}}
	     {\footnotemark\marginpar{\sffamily{\tiny
		   \addtocounter{footnote}{-1}\footnotemark#1
}\normalfont}}{}}

\newcommand{\showcomments}{yes}

\newcommand{\bel}[1]{\begin{equation}\label{#1}}
\newcommand{\ee}{\end{equation}}

\newcommand{\LBA}{\left\{\begin{array}}
\newcommand{\EAR}{\end{array}\right.}

\def\ovb{{\overline{b}}}

\def\NP{{\mathbf{NP}}}
\def\XP{{\mathbf{XP}}}
\def\DP{{\mathbf{DP}}}

\def\lamp{{L_2}}
\def\base{{\mathbb{Z}_2^{\mathbb{Z}}}}
\def\one{{\mathbf{1}}}

\newcommand{\gp}[1]{{\left\langle #1 \right\rangle}}
\newcommand{\gpr}[2]{{\left\langle #1 \mid #2 \right\rangle}}
\newcommand{\rb}[1]{{\left( #1 \right)}}

\newcommand{\Set}[2]{\left\{\, #1 \;\middle|\; #2 \,\right\}}

\def\MN{{\mathbb{N}}}
\def\MZ{{\mathbb{Z}}}

\DeclareMathOperator{\Aut}{{Aut}}

\DeclareMathOperator{\im}{{im}}

\DeclareMathOperator{\supp}{{supp}}
\DeclareMathOperator{\TPART}{3PART}

\DeclareMathOperator{\diam}{{diam}}

\usepackage[framemethod=TikZ]{mdframed}
\usepackage[listings,theorems,most]{tcolorbox}
\tcbset{before={\par\medskip\pagebreak[0]\noindent},after={\par\medskip}}
\usepackage{xcolor}

\title{Quadratic equations in the lamplighter group}
\author{Alexander Ushakov and Chloe Weiers}
\address{Department of Mathematical Sciences, Stevens Institute of Technology, Hoboken NJ 07030}\email{aushakov,cweiers@stevens.edu}
\thanks{The authors express their gratitude to the anonymous reviewer whose critique improved the quality of this article.}

\date{\today}

\begin{document}
\maketitle

\begin{abstract}
In this paper we study the complexity of solving quadratic equations in 
the lamplighter group.
We give a complete classification of cases 
(depending on genus and other characteristics of a given equation)
when the problem is $\NP$-complete or polynomial-time decidable.
We notice that the conjugacy problem can be solved in linear time.
Finally, we prove that the problem belongs to the class $\XP$.

\end{abstract}

\section{Introduction}

Let $F = F(Z)$ denote the free group on countably many generators 
$Z = \{z_i\}_{i=1}^\infty$. For a group $G$, an 
\emph{equation over $G$ with variables in $Z$} is an equality of the form $W = 1$, where $W \in F*G$. If $W = z_{i_1}g_1\cdots z_{i_k}g_k$, with $z_{i_j}\in Z$ and $g_j\in G$, then we refer to $\{z_{i_1},\ldots,z_{i_k}\}$ as the set of \emph{variables} and to $\{g_1,\ldots,g_k\}$ as the set of \emph{constants} (or \emph{coefficients}) of $W$. We occasionally write $W(z_1,\ldots,z_k)$ or $W(z_1,\ldots,z_k;g_1,\ldots,g_k)$ to indicate that the variables in $W$ are precisely $z_1,\ldots,z_k$ and (in the latter case) the constants are precisely $g_1,\ldots, g_k$. 

A \emph{solution} to an equation $W(z_1,\ldots,z_k)=1$ over $G$ is a homomorphism $$\varphi\colon F*G\rightarrow G $$
such that $\varphi|_G = 1_G$ and $W\in \ker \varphi$. If $\varphi$ is a solution of $W=1$ and $g_i = \varphi(z_i)$, then we often say that $g_1,\ldots,g_k$ is a solution of $W=1$. We also write $W(g_1,\ldots,g_k)$ or $W(\bar{g})$ to denote the image of $W(z_1,\ldots,z_k)$ under the homomorphism $F*G \rightarrow G$ which maps $z_i \mapsto g_i$ (and restricts to the identity on $G$). Note that while some authors allow for solutions in some overgroup $H$ (with an embedding $G\hookrightarrow H$), in this article we only consider solutions in $G$. 

In this paper we assume that $G$ comes equipped with a fixed generating set $X$
and elements of $G$ are given as products of elements of $X$ and their inverses.
This naturally defines the length (or size) of the equation $W=1$ as
the length of its left-hand side $W$.

\begin{definition}
An equation $W=1$ is called \emph{quadratic} if each variable appears exactly twice (as either $z_i$ or $z_i^{-1}$).
\end{definition}

The \emph{Diophantine problem} ($\DP$) in a group $G$ for a class of equations $C$ 
is an algorithmic question to decide whether a given equation $W=1$
in $C$ has a solution. In this paper we study the class of quadratic equations 
over the standard lamplighter group $L_2$.

\subsection{Classification of quadratic equations}

We say that equations $W = 1$ and $V=1$ are \emph{equivalent} 
if there is an automorphism $\phi\in \Aut(F\ast G)$ such that 
$\phi$ is the identity on $G$ and $\phi(W) = V$. It is a well 
known consequence of the classification of compact surfaces 
that any quadratic equation over $G$ is equivalent, via an 
automorphism $\phi$ computable in time $O(|W|^2)$, to an equation 
in exactly one of the following three \emph{standard forms} (see 
\cite{Comerford_Edmunds:1981,Grigorchuk-Kurchanov:1992}):
\begin{align}
\prod_{j=1}^k z_j^{-1} c_j z_j&=1 &k\ge 1,\label{eq:spherical}\\
\prod_{i=1}^g[x_i,y_i]\prod_{j=1}^k z_j^{-1} c_j z_j&=1 &g\geq 1, k\geq 0, \label{eq:orientable}\\
\prod_{i = 1}^g x_i^2\prod_{j=1}^k z_j^{-1} c_j z_j&=1 &g \geq 1, k \geq 0.\label{eq:nonorientable}
\end{align}
The number $g$ is the \emph{genus} of the equation, and both $g$ and $k$ 
(the number of constants) are invariants. 
The standard forms are called, respectively, 
\emph{spherical}, \emph{orientable of genus $g$}, and 
\emph{non-orientable of genus $g$}.

\subsection{Previous results}

The Diophantine problem for 
quadratic equations naturally generalizes fundamental (Dehn) problems
of group theory such as the word and conjugacy problems.
Furthermore, there is a deep relationship between quadratic equations 
and compact surfaces, which makes quadratic equations an interesting
object of study.

Study of quadratic equations originated with Malcev,
who in \cite{Malcev:1962} considered some particular
quadratic equations over free groups, 
now often referred to as \emph{Malcev's equations}.
This line of research for free groups was continued:
solution sets were studied  in \cite{Grigorchuk-Kurchanov:1992}, 
$\NP$-completeness was proved in 
\cite{Diekert-Robson:1999,Kharlampovich-Lysenok-Myasnikov-Touikan:2010}.
Various classes of infinite groups were analyzed:
hyperbolic groups were studied in
\cite{Grigorchuk-Lysenok:1992,Kharlampovich-Taam:2017},
the first Grigorchuk group was studied in
\cite{Lysenok-Miasnikov-Ushakov:2016} and \cite{Bartholdi-Groth-Lysenok:2022},
free metabelian groups were studied in
\cite{Lysenok-Ushakov:2015,Lysenok-Ushakov:2021},
metabelian Baumslag--Solitar groups were studied in
\cite{Mandel-Ushakov:2023b}.
The Diophantine problem for quadratic equations over the 
lamplighter group $L_2$ 
was recently shown to be decidable by Kharlampovich, Lopez, 
and Miasnikov 
(see \cite{Kharlampovich-Lopez-Miasnikov:2020}).

\subsection{History of the study of $L_2$}

The group $\lamp$ was first studied under the name of the lamplighter group by Kaimanovich and Vershik \cite{Vershik_Kaimanovich:1983}, though it has been an object of study for much longer as a simple example of a wreath product (see \cite{Bonanome_Dean_Dean_2018} for a compact history). The standard lamplighter group $L_2$ has many different realizations, including as a dynamical system, an infinite direct sum, and notably as a self-similar group generated by a $2$-state automaton (see Grigorchuk and Żuk \cite{Grigorchuk_Zuk:2001}). Furthermore, Skipper and Steinberg \cite{skipper2019lamplightergroupsbireversibleautomata} realized general lamplighter groups $A\wr\MZ$, where $A$ is finite, as automaton groups. The construction and mechanics of the lamplighter group extend naturally beyond $\lamp$ to the generalized lamplighter group $L_n$; dead-end elements of generalized lamplighter groups were described by Cleary and Taback in \cite{Cleary_Taback:2005}, and Bartholdi and Šunik showed that generalized lamplighter groups are self-similar \cite{Bartholdi:2006}.

Complexity of solving some fundamental problems has also been a topic of study in lamplighter groups. The conjugacy search problem was explored by Sale for general wreath products in \cite{Sale:2014b} and in lamplighter groups specifically in \cite{Sale:2016}. Siehler explored the conjugacy problem and described conditions for membership in the derived subgroup of $L_n$ in \cite{Siehler:2012}. In this paper we provide a different but equivalent description of the derived subgroup of $L_2$, which is naturally more compact due to the simplicity of $\MZ_2$.

\subsection{Model of computation and representation of basic objects}

We use a random-access machine as a model of computation.
Integers are given in unary form unless stated otherwise.
Elements of $L_2$ are given as words, i.e., sequences of letters
of the group alphabet $\{a^{\pm1},t^{\pm1}\}$ for $L_2$.

\subsection{Outline}

In Section \ref{se:lamplighter} we review the definition of the
lamplighter group $L_2$ as the wreath product $\MZ_2\wr \MZ$ and discuss basic computational properties of $L_2$.
In Section \ref{se:spherical} we prove that
the Diophantine problem for spherical equations is $\NP$-complete
and employ the Knuth-Morris-Pratt algorithm to design a linear time 
algorithm for the conjugacy problem in $L_2$.
In Section \ref{se:orientable} we design a linear time 
algorithm for orientable equations of genus $g\ge 1$.
In Section \ref{se:nonorientable} we prove that the problem is $\NP$-complete
for non-orientable quadratic equations of genus $g=1$ and linear time decidable
for non-orientable quadratic equations of genus $g\ge 2$.

\section{Preliminaries: the lamplighter group}
\label{se:lamplighter}

Define a set 
$$
\base = \Set{f\colon \MZ\to \MZ_2}{|\supp(f)|<\infty}
$$
and a binary operation $+$ on $\base$ which for $f,g\in \base$ produces $f+g\in \base$ defined by
$$
(f+g)(x) = f(x)+g(x) \ \mbox{ for } \ x\in \MZ.
$$
For a nontrivial $f$ it will be convenient to define
$$
m(f)=\min_{x\in \supp(f)} x
\ \ \mbox{ and }\ \ 
M(f)=\max_{x\in \supp(f)} x.
$$
For $f\in \base$ and $b\in \MZ$, define $f^b\in \base$ by
$$
f^b(x)=f(x+b) \
\mbox{ for } \ x\in \MZ,
$$
which is a right $\MZ$-action on $\base$ because $f^0=f$ and 
$f^{(b_1+b_2)}(x)=f(x+b_1+b_2)=(f^{b_1})^{b_2}(x)$.
Hence we can consider a semidirect product $\MZ \ltimes \base$
equipped with the operation
$$
(\delta_1,f_1)(\delta_2,f_2)=
(\delta_1+\delta_2,f_1^{\delta_2}+f_2).
$$
The group $\MZ \ltimes \base$ is called a \emph{restricted wreath product} of $\MZ_2$ and $\MZ$ and is denoted by $\MZ_2\wr\MZ$. The group $\MZ_2\wr\MZ$ is also known as the \textit{lamplighter group}, 
as it can be viewed as an infinite set of lamps (each lamp indexed by an element of $\MZ$), with each lamp either on or off, and a lamplighter positioned at some lamp.
Given some element $(\delta,f)\in\lamp$, $f\in\base$ represents the configuration of illuminated lamps and $\delta\in\MZ$ represents the position of the lamplighter. Similarly, we can define $\MZ_2\wr \MZ_n$ as $\MZ_n\ltimes \MZ_2^{\MZ_n}$. Henceforth, we denote $\MZ_2\wr\MZ$ as $L_2$, the generalized lamplighter group $\MZ_n\wr \MZ$ as $L_n$, and $\MZ_2\wr \MZ_n$ by $L_2^n$.

\subsection{Mechanics of $\lamp$}

There is a natural (abelian group) isomorphism between $\base$ and the ring 
$\MZ_2[x^\pm]$ of \emph{Laurent polynomials}
with coefficients in $\MZ_2$ that maps 
$f\in \base$ to $\sum_{i=-\infty}^\infty f(i)x^i$.
Hence, the group $L_2$ can be viewed 
as $\MZ\ltimes \MZ_2[x^\pm]$ with the $\MZ$-action on $\MZ_2[x^\pm]$
defined by
$$
f^\delta = f\cdot x^{-\delta}.
$$
Similarly, we can view
$\MZ_2^{\MZ_n}$ as the quotient ring $\MZ_2[x^\pm]/\gp{x^n-1}$
and $L_2^n$ as $\MZ_n\ltimes (\MZ_2[x^\pm]/\gp{x^n-1})$. Throughout this section we use notation for
$\base$ and $\MZ_2[x^\pm]$ interchangeably 
(respectively, $\MZ_2^{\MZ_n}$ and $\MZ_2[x^\pm]/\gp{x^n-1}$), slightly abusing the notation.

For $b\ge 0$ define a (group) homomorphism $\pi_b\colon \MZ_2[x^\pm] \to \MZ_2[x^\pm]/\gp{x^b-1}$ by
$$
[\pi_b(f)](x) = \sum_{y\equiv_b x} f(y) \mbox{ for } x\in\MZ_b.
$$
For $\ovb=(b_1,\ldots,b_k)\in \MZ^k$, define 
a linear function $\lambda_{\ovb}\colon (\base)^k\to \base$
by
$$
(f_1,\ldots,f_k)
\ \ \stackrel{\lambda_{\ovb}}{\mapsto}\ \ 
(1-x^{b_1})\cdot f_1+\cdots+(1-x^{b_k})\cdot f_k,
$$ 
$b=\gcd(b_1,\ldots,b_k)$, 
and a projection homomorphism 
$\pi_\ovb=\pi_b\colon \base\to \MZ_2^{\MZ_b}$.
Throughout the paper we use the conventions that $\gcd(0,\dots,0)=0$ and $\MZ_0=\MZ$.

\begin{lemma}\label{le:poly_gcd}
If $b_1,b_2>0$
and $b=\gcd(b_1,b_2)$, then $\gcd(x^{b_1}-1,x^{b_2}-1)=x^{b}-1$.
\end{lemma}

\begin{proof}
The statement follows from the Euclidean lemma.
Indeed, if $b_2=b_1+r$, then
$$
(x^{b_2}-1) - x^r(x^{b_1}-1) =  x^r-1,
$$
and, hence,
$\gcd(x^{b_1}-1,x^{b_2}-1)=\gcd(x^{b_1}-1,x^{r}-1)$.
That allows us to define a Euclidean-like process that terminates with $x^b-1$.
\end{proof}

\begin{cor} \label{cor:sequencemulti}
$\im(\lambda_{\ovb}) = \ker(\pi_{\ovb})$.
\end{cor}

\begin{proof}
A ring of Laurent polynomials over a field, such as $\MZ_2[x^\pm]$,
is a principal ideal domain.
By definition of $\lambda_{\ovb}$, $\im(\lambda_{\ovb})$ is the ideal generated 
by $x^{b_1}-1,\dots,x^{b_k}-1$ in $\MZ_2[x^\pm]$.
It follows from Lemma \ref{le:poly_gcd} that $\im(\lambda_{\ovb}) = \gp{x^b-1}$.
Clearly, $\ker(\pi_{\ovb})=\ker(\pi_{b})$ is the same ideal $\gp{x^b-1}$.
\end{proof}

\begin{proposition}\label{pr:transform-components}
Let
$W=\sum_{i=1}^k (1-x^{b_i})\cdot f_i + W'(f_{k+1},\dots)$,
where $b_i\in \MZ$ are constants,
$f_1,\ldots,f_k,f_{k+1},\ldots\in \base$ are variables, and 
$W'$ does not involve $f_1,\dots,f_k$.
Then
$$
W=0 \mbox{ has a solution in } \base
\ \ \Leftrightarrow\ \ 
W'=0\mbox{ has a solution in } \MZ_2^{\MZ_b},
$$
where $b=\gcd(b_1,\dots,b_k)$.
\end{proposition}

\begin{proof}
``$\Rightarrow$'' 
Suppose that $W=0$ has a solution in $\base$. 
Applying a homomorphism $\pi_\ovb$ we get the equation $\pi_{\ovb}(W)=0$
in which the term $\sum_{i=1}^k (1-x^{b_i})\cdot f_i$ vanishes
and that has a solution in $\MZ_2^{\MZ_b}$.

``$\Leftarrow$''
Suppose that $W'=0$ has a solution $f_{k+1}',f_{k+2}',\ldots \in \MZ_2^{\MZ_b}$. 
Choose any functions $f_{k+1}^\ast,f_{k+2}^\ast,\ldots\in \base$ 
satisfying $f_i'=\pi_\ovb(f_i^\ast)$. 
Then, since $\pi_\ovb$ is a homomorphism, for some 
$f_1^\ast,\dots,f_k^\ast \in \base$ we have
\begin{align*}
\pi_\ovb(W'(f_{k+1}^\ast,\dots))=
W'(f_{k+1}',\dots)=0
&\ \ \Rightarrow\ \ 
W'(f_{k+1}^\ast,\dots) \in \ker(\pi_\ovb)\stackrel{\ref{cor:sequencemulti}}{=} \im(\lambda_\ovb)\\
&\ \ \Rightarrow\ \ 
W'(f_{k+1}^\ast,\dots) = \sum_{i=1}^k (1-x^{b_i})\cdot f_i^\ast\\
&\ \ \Rightarrow\ \ 
\sum_{i=1}^k (1-x^{b_i})\cdot (-f_i^\ast) + W'(f_{k+1}^\ast,\dots) = 0.
\end{align*}
Thus, $W=0$ has a solution.
\end{proof}

\subsection{Several useful facts about $L_2$}
Let
$a=(0,\mathbf{1}_0)$, where $\one_0 \in \MZ_2^\MZ$ is defined by
$$
\one_0(x)=
\begin{cases}
1& \mbox{ if } x=0\\
0& \mbox{ if } x\ne0\\
\end{cases}
$$
and $t = (1,0)$, where the $0$ in the second component is
the function identical to $0$.
It is easy to see that $L_2$ is generated by $a$ and $t$, i.e.,
every element of $L_2$ can be defined as a word in the alphabet $\{a,t\}$.
It is a well-known fact that $L_2$
is not finitely presented, see \cite{Baumslag:1961}.

For a given group word $w=w(a,t)$ it is straightforward to
compute $\delta\in\MZ$ and $f\colon\MZ\to \MZ_2$ such that
the group element defined by $w$ is the same as $(\delta,f)$.
The value of $\delta$ is the sum of exponents of $t$ in $w$
and is kept in unary. The function $f$ is represented by its values
over the range $m=m(f),\dots,M=M(f)$ and is kept as a sequence of bits
$f(m),\dots,f(M)$ together with $m$ and $M$ in unary.
More specifically, we use a \emph{double-ended queue}, 
abbreviated to \emph{deque},
to ensure that the representation of $f$
is expandable left and right in linear time when needed.
If $\supp(f)=\varnothing$, then $f$ is represented by
the sequence $0$.

Let $\sigma_a\colon L_2\to \MZ_2$ and $\sigma_t\colon L_2\to \MZ$ 
be ``projection homomorphisms'' defined by
$$
\sigma_a(g) = \sum_{i=-\infty}^\infty f(i) 
\ \ \mbox{ and }\ \ 
\sigma_t(g) = \delta
$$
if $g=(\delta,f)$. 
For a given word $w=w(a,t)$ in the alphabet of $L_2$
the function $\sigma_a$ counts the sum of exponents of $a$ symbols
in $w$ modulo $2$ and the function $\sigma_t$ counts the sum of exponents of 
$t$ symbols in $w$.
Clearly, both $\sigma_a(w)$ and $\sigma_t(w)$ 
are computable in linear time $O(|w|)$.
The subgroup 
$$
Y=\Set{g\in L_2}{g=(0,f)}=\ker(\sigma_t)
$$
plays a special role in our complexity analysis in 
Sections \ref{se:spherical}, \ref{se:orientable}, and \ref{se:nonorientable}.

Finally, let us mention several useful formulae used throughout the paper:
\begin{align*}
(\delta,f)^{-1} &=(-\delta,-f^{-\delta}) =(-\delta,-x^{\delta}f),\\
(\delta,f)^2 &=(2\delta,f^{\delta}+f) =(2\delta,(x^{-\delta}+1)f),\\
(\delta,f)^{-1} (\delta_1,f_1) (\delta,f) 
&= 
(-\delta,-f^{-\delta})(\delta_1,f_1) (\delta,f)
= (\delta_1,(1-x^{-\delta_1})f+x^{-\delta}f_1),
\end{align*}
and make a simple observation that will be used multiple times
throughout the paper without explicit reference.

\begin{lemma}\label{le:Y-conjugation}
For every $g=(0,f_1)\in Y$ and $z=(\delta,f)\in L_2$,
$$
z^{-1}gz=
(\delta,f)^{-1} (0,f_1) (\delta,f) = 
(\delta,0)^{-1} (0,f_1) (\delta,0) = (0,f_1^{\delta}).
$$
\end{lemma}

\section{Spherical equations}
\label{se:spherical}

In this section we study the Diophantine problem for spherical equations ($\DP_{SPH}$),
that can be formally defined as follows.
\begin{algproblem}
 \problemtitle{\textsc{The Diophantine problem for spherical equations} $(\DP_{SPH})$}
  \probleminput{Words $c_1,\dots,c_k$ in the group alphabet $\{a,t\}$ of $L_2$.}
  \problemquestion{Does equation \eqref{eq:spherical} with coefficients $c_1,\dots,c_k$ 
  have a solution?
  }
\end{algproblem}

The problem for orientable (and non-orientable) equations is defined in a similar way.

\begin{algproblem}
 \problemtitle{\textsc{The Diophantine problem  for orientable quadratic equations} $(\DP_{ORIENT})$}
  \probleminput{Words $c_1,\dots,c_k$ in the group alphabet $\{a,t\}$ of $L_2$ and $g\in\MN$.}
  \problemquestion{Does equation \eqref{eq:orientable} with coefficients $c_1,\dots,c_k$  have a solution?
  }
\end{algproblem}


We use the $3$-partition problem $(\TPART)$ as a basis for $\NP$-hardness.
Let $S=\{s_1,\ldots,s_{3k}\}$ be a multiset of $3k$ positive integers
given in unary and
$$
T_S = \frac{1}{k}\sum_{i=1}^{3k}s_i.
$$

\begin{algproblem}
 \problemtitle{\textsc{$3$-partition problem} $(\TPART)$}
  \probleminput{Multiset $S=\{s_1,\ldots,s_{3k}\}$ given in unary, 
  satisfying $T_S/4 < s_i < T_S/2$.}
  \problemquestion{Is there a partition of $S$ into $k$ triples, each of which sums to $T_S$?
  }
\end{algproblem}
The problem $\TPART$ in this form is known to be $\NP$-complete.\footnote{A thorough treatment of this problem may be found in \cite{Garey-Johnson:1979}.}

\subsection{Complexity lower bound}

Consider the following restriction of $\DP_{SPH}$.

\begin{algproblem}
 \problemtitle{\textsc{$\DP$ for spherical equations with coefficients from $Y$} $(\DP_{SPH_0})$}
  \probleminput{Words $c_1,\dots,c_k$ defining elements of $Y$.}
  \problemquestion{Does equation \eqref{eq:spherical} with coefficients $c_1,\dots,c_k$ 
  have a solution?
  }
\end{algproblem}

Below we apply the method from \cite{Mandel-Ushakov:2023b}
to reduce $\TPART$ to $\DP_{SPH_0}$ in polynomial time. 
Let $S=\{s_1,\ldots,s_{3k}\}$ be an instance of $\TPART$ and 
$T = \frac{1}{k}\sum s_i$ be the anticipated sum for subsets, where 
$T/4 < s_i < T/2$ for each $i$.
For $y\in \MN$ define elements
$$
c_y= \prod_{i=0}^{y-1} t^i a t^{-i}, \ \ \ \  
c= \prod_{i=0}^{k-1} 
t^{(T+1)i}
c_T
t^{-(T+1)i}.
$$
Notice that $c_y$ and $c$ belong to $Y$.

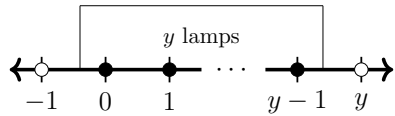
\begin{figure}[!h]
\centering
\scalebox{0.85}{
\begin{tikzpicture}
	\draw[<-,ultra thick] (-1.5,0)--(1.5,0);
        \draw[->,ultra thick] (2.5,0)--(4.5,0);
	\draw[-, thick] (-1,-0.2)--(-1,0.2);
	\draw[-, thick] (0,-0.2)--(0,0.2);
	\draw[-, thick] (4,-0.2)--(4,0.2);
	\draw[-, thick] (3,-0.2)--(3,0.2);
	\draw[-, thick] (1,-0.2)--(1,0.2);
 
	\node[draw=none] at (0,-0.5) {$0$};
        \node[draw=none] at (1,-0.5) {$1$};
        \node[draw=none] at (2,0) {$\cdots$};
        \node[draw=none] at (3,-0.5) {$y-1$};
        \node[draw=none] at (4,-0.5) {$y$};
        \node[draw=none] at (-1,-0.5) {$-1$};
	
	\filldraw[draw=black,fill=black] (1,0) circle (3pt);
        \filldraw[draw=black,fill=white] (-1,0) circle (3pt);
	\filldraw[draw=black,fill=black] (0,0) circle (3pt);
	\filldraw[draw=black,fill=black] (3,0) circle (3pt);
	\filldraw[draw=black,fill=white] (4,0) circle (3pt);
        \node[draw=none] at (1.5,0.5) {\footnotesize{$y$ lamps}};

        \draw [-] (-0.4,0) to (-0.4,1);
        \draw [-] (3.4,0) to (3.4,1);
        \draw [-] (-0.4,1) to (3.4,1);
        
	\end{tikzpicture}}
\caption{The ``lamp-configuration'' $f$ for the element $c_y$.}\label{figcy}
\end{figure}

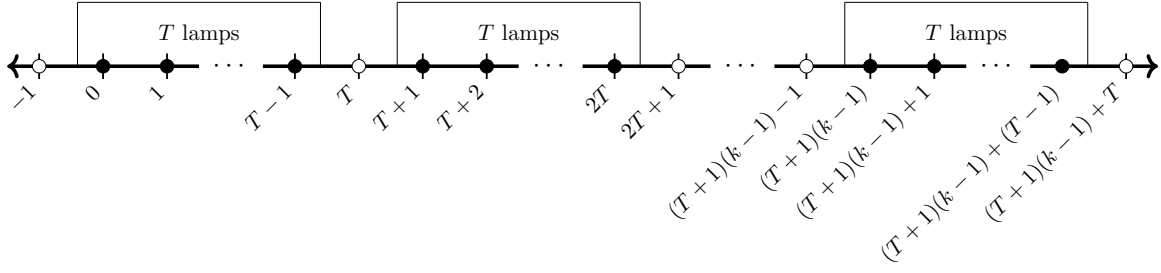
\begin{figure}[!h]
\centering
\scalebox{0.85}{
\begin{tikzpicture}
	\draw[<-,ultra thick] (-1.5,0)--(1.5,0);
        \draw[-,ultra thick] (2.5,0)--(6.5,0);
        \draw[-,ultra thick] (7.5,0)--(9.5,0);
        \draw[-,ultra thick] (10.5,0)--(13.5,0);
        \draw[->,ultra thick] (14.5,0)--(16.5,0);
	\draw[-, thick] (-1,-0.2)--(-1,0.2);
	\draw[-, thick] (0,-0.2)--(0,0.2);
	\draw[-, thick] (5,-0.2)--(5,0.2);
	\draw[-, thick] (4,-0.2)--(4,0.2);
	\draw[-, thick] (3,-0.2)--(3,0.2);
	\draw[-, thick] (1,-0.2)--(1,0.2);
        \draw[-, thick] (6,-0.2)--(6,0.2);
        \draw[-, thick] (8,-0.2)--(8,0.2);
        \draw[-, thick] (9,-0.2)--(9,0.2);
        \draw[-, thick] (11,-0.2)--(11,0.2);
        \draw[-, thick] (12,-0.2)--(12,0.2);
        \draw[-, thick] (13,-0.2)--(13,0.2);
        \draw[-, thick] (16,-0.2)--(16,0.2);

        \node[draw=none, rotate=45, anchor=east] at (-1,-0.25) {\footnotesize$-1$};
	\node[draw=none, rotate=45, anchor=east] at (0,-0.25) {\footnotesize$0$};
        \node[draw=none, rotate=45, anchor=east] at (1,-0.25) {\footnotesize$1$};
        \node[draw=none] at (1.5,0.5) {\footnotesize{$T$ lamps}};
        \node[draw=none] at (2,0) {$\cdots$};
        \node[draw=none, rotate=45, anchor=east] at (3,-0.25) {\footnotesize$T-1$};

        \node[draw=none, rotate=45, anchor=east] at (4,-0.25) {\footnotesize$T$};
        \node[draw=none, rotate=45, anchor=east] at (5,-0.25) {\footnotesize$T+1$};
        \node[draw=none, rotate=45, anchor=east] at (6,-0.25) {\footnotesize$T+2$};
        \node[draw=none] at (6.5,0.5) {\footnotesize{$T$ lamps}};
        \node[draw=none] at (7,0) {$\cdots$};
        \node[draw=none, rotate=45, anchor=east] at (8,-0.25) {\footnotesize$2T$};
        \node[draw=none, rotate=45, anchor=east] at (9,-0.25) {\footnotesize$2T+1$};

        \node[draw=none] at (10,0) {$\cdots$};
        \node[draw=none, rotate=45, anchor=east] at (11,-0.25) {\footnotesize$(T+1)(k-1)-1$};
        \node[draw=none, rotate=45, anchor=east] at (12,-0.25) {\footnotesize$(T+1)(k-1)$};
        \node[draw=none, rotate=45, anchor=east] at (13,-0.25) {\footnotesize$(T+1)(k-1)+1$};
        \node[draw=none] at (14,0) {$\cdots$};
        \node[draw=none] at (13.5,0.5) {\footnotesize{$T$ lamps}};
        \node[draw=none, rotate=45, anchor=east] at (15,-0.25) {\footnotesize$(T+1)(k-1)+(T-1)$};
        \node[draw=none, rotate=45, anchor=east] at (16,-0.25) {\footnotesize$(T+1)(k-1)+T$};

	\filldraw[draw=black,fill=black] (1,0) circle (3pt);
	\filldraw[draw=black,fill=black] (3,0) circle (3pt);
        \filldraw[draw=black,fill=white] (-1,0) circle (3pt);
	\filldraw[draw=black,fill=black] (0,0) circle (3pt);
	\filldraw[draw=black,fill=white] (4,0) circle (3pt);
	\filldraw[draw=black,fill=black] (5,0) circle (3pt);
        \filldraw[draw=black,fill=black] (6,0) circle (3pt);
        \filldraw[draw=black,fill=black] (8,0) circle (3pt);
        \filldraw[draw=black,fill=white] (9,0) circle (3pt);
        \filldraw[draw=black,fill=white] (11,0) circle (3pt);
        \filldraw[draw=black,fill=black] (12,0) circle (3pt);
        \filldraw[draw=black,fill=black] (13,0) circle (3pt);
        \filldraw[draw=black,fill=black] (15,0) circle (3pt);
        \filldraw[draw=black,fill=white] (16,0) circle (3pt);

        \draw [-] (-0.4,0) to (-0.4,1);
        \draw [-] (3.4,0) to (3.4,1);
        \draw [-] (-0.4,1) to (3.4,1);

        \draw [-] (4.6,0) to (4.6,1);
        \draw [-] (8.4,0) to (8.4,1);
        \draw [-] (4.6,1) to (8.4,1);

        \draw [-] (11.6,0) to (11.6,1);
        \draw [-] (15.4,0) to (15.4,1);
        \draw [-] (11.6,1) to (15.4,1);
	\end{tikzpicture}
}
\caption{
The ``lamp-configuration'' $f$ for the element
$c$ consists of $k$ clusters of $T$ lit lamps, each cluster separated by a single unlit lamp.}\label{figc}
\end{figure}

\begin{proposition}\label{pr:tpart}
$S=\{s_1,\ldots,s_{3k}\}$ is a positive instance of $\TPART$ if and only if
\begin{equation}\label{eq:spherical_special}
    \prod_{i=1}^{3k} z_i^{-1}c_{s_i}z_i=c
\end{equation}
has a solution.
\end{proposition}

\begin{proof}
``$\Rightarrow$''
If $S$ is a positive instance of $\TPART$, then, without loss of generality, 
we may assume that $\sum_{i=1}^3 s_{3j+i}=T$ for $j=0,1,\ldots,k-1$.
Let $T_i=(T+1)i$ for $1\le i\le 3k$. Then clearly $z_i=t^{\Delta_i}$ for $i=1,\ldots,3k$ is a solution to \eqref{eq:spherical_special}, where $\Delta_i$'s are defined as:
\begin{align*}
\Delta_1&=0 &\Delta_2&=-s_1 &\Delta_3&=-s_1-s_2\\
\Delta_4&=-T_1 &\Delta_5&=-T_1-s_4 &\Delta_6&=-T_1-s_4-s_5\\
\Delta_7&=-T_2 &\Delta_8&=-T_2-s_7 &\Delta_9&=-T_2-s_7-s_8\\
&\cdots &&\cdots &&\cdots 
\end{align*}

``$\Leftarrow$''
Conversely, assume there exists an assignment $z_i=(\delta_i,f_i)$
satisfying \eqref{eq:spherical_special}. By design,
$c_i\in Y$ for every $i=1,\dots,3k$.
Hence, using Lemma \ref{le:Y-conjugation}, we may assume that 
$z_i=(\delta_i,0)$ and can translate \eqref{eq:spherical_special} 
into the equality $\sum_{i=1}^{3k} f_{c_i}^{\delta_i} = f_c$, which, in particular, 
implies that
$$
\supp\Bigl(\sum_{i=1}^{3k}f_{c_i}^{\delta_i}\Bigr) = \supp(f_c).
$$
Notice that 
$$
\Bigl|\supp\Bigl(\sum_{i=1}^{3k}f_{c_i}^{\delta_i}\Bigr)\Bigr|
\ \le\ 
\sum_{i=1}^{3k} |\supp(f_{c_i}^{\delta_i})|=
\sum_{i=1}^{3k} |\supp(f_{c_i})|=
kT = |\supp(f_c)|,
$$
where the first inequality is strict if and only if 
$\supp(f_{c_i}^{\delta_i}) \cap \supp(f_{c_j}^{\delta_j}) \ne \varnothing$
for distinct $i,j$. Therefore, the sets
$\supp(f_{c_1}^{\delta_1}),\dots,\supp(f_{c_{3k}}^{\delta_{3k}})$
are pairwise disjoint and, viewed geometrically, form a tiling of $\supp(f_c)$.

As shown in Figure \ref{figc}, $\supp(f_c)$ consists of $k$ sequences
of integers each of length $T$ 
$$
\supp(f_c)=
\underbrace{\{0,\dots,T-1\}}_{S_1}\ \sqcup\ 
\underbrace{\{T+1,\dots,2T\}}_{S_2}\ \sqcup\ 
\underbrace{\{2T+2,\dots,3T+1\}}_{S_3} \ \sqcup\ 
\cdots
$$
separated by single missing numbers
between those sequences. On the other hand, the sets
$\supp(f_{c_i}^{\delta_i})$ have no gaps.
Therefore, shifts $\delta_1,\dots,\delta_{3k}$ tile
each sequence $S_l$ in $\supp(f_c)$ by sets $\supp(f_{c_i}^{\delta_i})$, i.e.,
$$
S_l = 
\supp(f_{c_i}^{\delta_i})\ \sqcup\ 
\supp(f_{c_j}^{\delta_j})\ \sqcup\ \dots \sqcup\ 
\supp(f_{c_k}^{\delta_k}),
$$
which implies that 
$$
T=|S_l| = |\supp(f_{c_i}^{\delta_i})|+
|\supp(f_{c_j}^{\delta_j})|+\dots+
|\supp(f_{c_k}^{\delta_k})|=
c_i+c_j+\dots+c_k
$$
and proves that $S$ is a positive instance of $\TPART$
(it is easy to see that the condition $T/4 < s_i < T/2$ for each $i$
imposed on $S$ implies that each $T$ is expressed as a sum of exactly three
$c_i$'s).
\end{proof}

\begin{corollary}\label{co:SPH-0-hard}
$\DP_{SPH_0}$ (and $\DP_{SPH}$) is $\NP$-hard.
\end{corollary}

\begin{proof}
All coefficients in \eqref{eq:spherical_special} belong to the subgroup $Y$.
Therefore,
the reduction discussed in Proposition \ref{pr:tpart}
is a polynomial-time Karp reduction from $\TPART$ to $\DP_{SPH_0}$.
\end{proof}

\subsection{Complexity upper bound}

Now we prove that the Diophantine problem for spherical equations \eqref{eq:spherical}
over $L_2$ belongs to $\NP$. 

\begin{proposition}\label{pr:spherical-reduction}
A spherical equation \eqref{eq:spherical} with constants 
$c_i=(\delta_{c_i},f_{c_i})$ has a solution if and only if
the following two conditions are satisfied:
\begin{itemize}
\item[(Sp1)]
$\sum_{i=1}^k \delta_{c_i} = 0$,
\item[(Sp2)]
$\sum_{i=1}^k f_{c_i}^{\delta_i}= 0$ in $\MZ_2^{\MZ_\delta}$
for some $\delta_1,\dots,\delta_k\in \MZ_\delta$,
\end{itemize}
where $\delta=\gcd(\delta_{c_1},\ldots,\delta_{c_m})$.
\end{proposition}

\begin{proof}
Let $z_i=(\delta_{z_i},f_{z_i})$.
Direct computation shows that $z_1,\dots,z_k$
satisfy equation \eqref{eq:spherical} if and only if
$$
\rb{
\sum_{i=1}^k \delta_{c_i},
\sum_{i=1}^k x^{-\Delta_i}\left[
\rb{1-x^{-\delta_{c_i}}}f_{z_i}+
x^{-\delta_{z_i}}f_{c_i}
\right]
}=(0,0)
$$
where $\Delta_i=\delta_{c_{i+1}}+\dots+\delta_{c_k}$ for $i=1,\dots,k$,
is satisfied.
That further translates to the conditions $\sum_{i=1}^k \delta_{c_i} = 0$
(which establishes (Sp1)) and
$$
\sum_{i=1}^k x^{-\Delta_i}\left[
\rb{1-x^{-\delta_{c_i}}}f_{z_i}+
x^{-\delta_{z_i}}f_{c_i}
\right]
= 0\ \mbox{ in }\ \MZ_2^\MZ.
$$
The latter condition is satisfied by 
$f_{z_i}\in \MZ_2^\MZ$ and $\delta_{z_i}\in\MZ$
if and only if the condition
\begin{equation}\label{eq:condition-step2}
\sum_{i=1}^k \left[
\rb{1-x^{-\delta_{c_i}}} f_{i}+
x^{\delta_{i}}f_{c_i}
\right]
= 0\ \mbox{ in }\ \MZ_2^\MZ
\end{equation}
is satisfied
by $f_i=x^{-\Delta_i}f_{z_i}$ and $\delta_{i}=-\delta_{z_i}-\Delta_i$.
By Proposition \ref{pr:transform-components},
for fixed $\delta_1,\dots,\delta_k\in\MZ$,
equation \eqref{eq:condition-step2} has a solution
if and only if
$\sum_{i=1}^k x^{\delta_{i}}f_{c_i} = 0$
in $\MZ_2^{\MZ_\delta}$.
Hence, equation \eqref{eq:condition-step2} has a solution
if and only if 
$\sum_{i=1}^k x^{\delta_{i}}f_{c_i} = 0$
holds for some $\delta_1,\dots,\delta_k\in\MZ_\delta$.
This proves the proposition.
\end{proof}

For a nontrivial function $f\colon\MZ\to \MZ_2$ define 
$$
\diam(f)= M(f)-m(f),
$$
and note that $\diam(f)$ is invariant under a $\delta$-shift, i.e.,
$\diam(f^{\delta}) = \diam\rb{f}$ for every $\delta\in\MZ$.
The next lemma states that the diameter of $f$ cannot exceed
the length of a shortest word $w$ describing that function.

\begin{lemma}\label{le:inv2}
For any word $w=w(a,t)$ in generators of $L_2$
defining the element $(\delta,f)$,
$$
|\delta|\le |w|
\ \mbox{ and }\ 
\diam\rb{f}\leq \left|w\right|.
$$
Furthermore,
$\diam\rb{f}\leq \left|w\right|/2-1$ for any $w=w(a,t)$ defining $(0,f)$.
\end{lemma}

\begin{theorem}\label{th:spherical-NP}
The Diophantine problem for spherical equations over $\lamp$ belongs to $\NP$.
\end{theorem}

\begin{proof}
Consider a spherical equation $W=1$ of type \eqref{eq:spherical} that has a solution.
By Proposition \ref{pr:spherical-reduction}, it has a solution if and only if 
the conditions (Sp1) and (Sp2) hold.
The condition (Sp1) can be verified in linear time.
The condition (Sp2) is an equation for $\delta_{1},\dots,\delta_{k}\in\MZ_\delta$.
We claim that if (Sp2) has a solution, then it has a bounded solution
satisfying 
$$
-|W|\le \delta_{i} \le |W|,
$$
for every $i=1,\dots,k$,
which gives an efficiently verifiable certificate for \eqref{eq:spherical}.

If $\delta\ge 1$, then 
\begin{equation}\label{eq:delta_i-bound-mod-W}
0\le \delta_{i}<\delta\le |W|
\end{equation}
and the claim holds.
Therefore, we may assume that $\delta=0$.
Fix $\delta_1,\ldots,\delta_k$ satisfying (Sp2). 
Let $I=\{1,\ldots,k\}$ be the set of indices.
Define an equivalence relation $R$ as the transitive closure of the following binary relation on $I$:
$$
J=\Set{(i,j)\in I\times I}{\supp(f_{c_i}^{\delta_i})\cap\supp(f_{c_j}^{\delta_j})\neq\varnothing}.
$$
A \emph{cluster} for the sequence of functions 
$f_{c_1},\ldots,f_{c_k}$ is an equivalence class 
for the relation $R$. We can partition $I$ into a finite union of some $n$ 
clusters $C_i$, i.e. $I=C_1\sqcup \cdots\sqcup C_n$.
Note that $\sum_{i\in C}f_{c_i}^{\delta_i}=0$ for every cluster $C$. 
Hence, shifting a cluster independently of other clusters does not
change the sum of its shifted component functions. In other words,
\[
\sum_{i\notin C} f_{c_i}^{\delta_i}+\sum_{i\in C} f_{c_i}^{\delta_i+\Delta}=0,
\]
for any cluster $C$ and any $\Delta\in\MZ$. Define the support of a cluster $C$ as 
$$
\supp(C) = \bigcup_{i\in C} \supp\rb{f_{c_i}^{\delta_i}}
$$
(note that $\supp(C)\subseteq \MZ$ because $\delta=0$)
and its diameter as
\[
\diam(C)= \max \Set{x-y}{x,y\in\supp(C)}.
\]
It follows from the definition of a cluster that the diameter of a cluster $C$ 
cannot exceed the sum of the diameters of its component functions, so
$$
\diam(C)
\leq
\sum_{i\in C} \diam\rb{f_{c_i}^{\delta_i}} 
=
\sum_{i\in C} \diam\rb{f_{c_i}} 
\stackrel{\ref{le:inv2}}{<} 
\sum_{i\in C} \left|c_i\right|/2\le |W|/2.
$$
Thus, the diameter of a single cluster is bounded by the lengths of the words describing its component functions.

For each cluster $C_j$ choose a ``midpoint'' for its support
$$
p_j=\left\lfloor \tfrac{1}{2}\rb{\min \Set{x}{x\in\supp(C_j)} + \max \Set{x}{x\in\supp(C_j)}} \right\rfloor
$$
and for $i=1,\dots,k$ define
$$
\delta_i'=\delta_i-p_j\  \mbox{ if } \ i\in C_j.
$$
Since $\delta_i$ from the same cluster are shifted the same way,
we have that $\delta_1',\dots,\delta_k'$ satisfies (Sp2).
Furthermore, this transformation ``centers'' cluster supports
around the origin and, if we denote by $C_j^\ast$
the \textbf{shifted} cluster $C_j$, then
for every $i$ in that cluster we have that
\begin{equation}\label{eq:cenetered-cluster-support}
\supp\big(f_{c_i}^{\delta_i'}\big)
\ \subseteq\ 
\supp(C_j^\ast) 
\ \subseteq\ 
\left[-\tfrac{|W|}{4},\tfrac{|W|}{4}\right].
\end{equation}
Since
$$
\supp(f_{c_i}) 
\ \subseteq\ 
\left[-\tfrac{|c_i|}{2},\tfrac{|c_i|}{2}\right]
\ \subseteq\ 
\left[-\tfrac{|W|}{2},\tfrac{|W|}{2}\right]
$$
it follows that $|\delta_i'|$ cannot be greater than 
$\tfrac{|W|}{2}+\tfrac{|W|}{4}\le |W|$, because
otherwise $\delta_i'$ shifts the interval 
$\left[-|W|/2,|W|/2\right]$ containing $\supp(f_{c_i})$ 
completely outside of the target window $\left[-|W|/4,|W|/4\right]$.
Therefore, $|\delta_i'|$ is bounded by $|W|$ as claimed.

Thus, in both cases, for a positive instance of $\DP_{SPH}$ there
exists a bounded sequence $\delta_1,\dots,\delta_k$ satisfying (Sp2)
which can be checked in polynomial time and which establishes an 
$\NP$-certificate for the instance.
\end{proof}

\begin{theorem}
The Diophantine problem for spherical equations over $\lamp$ is $\NP$-complete.
\end{theorem}

\begin{corollary}\label{co:equation1-complexity}
There is an algorithm that, for a given spherical equation \eqref{eq:spherical}, decides if it has a solution or not in time $O(|W|^k)$.
\end{corollary}

\begin{proof}
By Proposition \ref{pr:spherical-reduction}, \eqref{eq:spherical} has a solution if and only if (Sp1) and (Sp2) are satisfied.
(Sp1) can be checked in linear time.
Below, we discuss (Sp2) for two cases, $\delta=0$ and $\delta\ne 0$.

\textsc{(Case-I)} Suppose that $\delta=0$.
In the proof of Theorem \ref{th:spherical-NP} we showed that if
(Sp2) is satisfied by $\delta_1,\dots,\delta_k\in\MZ$, then 
it is satisfied by $\delta_1',\dots,\delta_k'\in\MZ$ 
satisfying \eqref{eq:cenetered-cluster-support}, i.e.,
$$
\supp\big(f_{c_i}^{\delta_i'}\big)\subseteq \left[-|W|/4,|W|/4\right],
$$
which gives at most $|W|/2$ distinct values for $\delta_i'$. Denote the set of values for $\delta_i'$ by $V_i$.
Now the procedure can be organized as follows:
\begin{itemize}
\item 
For each tuple
$(\delta_1',\dots,\delta_{k-1}') \in V_1\times\dots\times V_{k-1}$:
\begin{itemize}
\item 
Compute $f = \sum_{i=1}^{k-1} f_{c_i}^{\delta_i'}$.
\item 
Test if $f = f_{c_k}^{\delta_k'}$ for some $\delta_k'$
by computing $m_1=m(f)$ and $m_2=m(f_{c_k})$ and then simultaneously scanning 
$f$ and $f_{c_k}$ pairs of elements starting from positions $m_1$ and $m_2$ respectively. 
\end{itemize}
\end{itemize}
The number of tuples $(\delta_1',\dots,\delta_{k-1}')$ is bounded by $|W|^{k-1}$.
Computation of $f$ and 
testing if $f = f_{c_k}^{\delta_k'}$ can be done in $O(|W|)$ time.
Hence the claimed time bound.

\textsc{(Case-II)} If $\delta\ne 0$ and the equation has a solution, then,
as in Theorem \ref{th:spherical-NP},
$0\le \delta\le |W|$ 
and each $\delta_i$ can take at most $|W|$ distinct values.
Clearly, in this case we may assume that $\delta_k=0$ and
the procedure can be organized as follows:
\begin{itemize}
\item 
For each tuple
$(\delta_1,\dots,\delta_{k-1}) \in \left[0,|W|\right)^{k-1}$:
\begin{itemize}
\item 
Compute $f = \sum_{i=1}^{k-1} f_{c_i}^{\delta_i}$.
\item 
Test if $\pi_\delta(f) = \pi_\delta(f_{c_k})$. 
\end{itemize}
\end{itemize}
The complexity of this algorithm is also bounded by $O(|W|^k)$.
\end{proof}

\subsection{The conjugacy problem over $L_2$}
In \cite{Sale:2016}, Sale showed that the conjugacy search problem can be solved in quadratic time over $L_n$. Below we show that it can be solved in linear time for $L_2$.

\begin{proposition}\label{pr:conjugacy-linear}
The conjugacy problem over $L_2$ is decidable in linear time.
\end{proposition}

\begin{proof}
    First, observe that a spherical equation with two conjugates
    $$
    z_1^{-1}c_1z_1z_2^{-1}c_2z_2=1 \ \ \Rightarrow \ \ z_1^{-1}c_1z_1=z_2^{-1}c_2^{-1}z_2
    $$
    is an instance of the conjugacy problem with elements $c_1$ and $c_2^{-1}$. Let $c_1=(\delta_{c_1},f_{c_1})$ and $c_2^{-1}=(\delta_{c_2},f_{c_2})$ for notational convenience. For such an equation to have a solution, clearly $\delta_{c_1}=\delta_{c_2}$ must be satisfied. The problem thus splits into two cases, depending on whether $\delta$'s are trivial. In both cases, we exclude the trivial support case, where one or both functions satisfy $\supp(f_{c_i})=\varnothing$.
    
\textsc{(Case-I)} The first case is where $\delta_{c_1}=\delta_{c_2}=0$. By Proposition \ref{pr:spherical-reduction}, (Sp1) and (Sp2) must be satisfied. The condition (Sp1) is satisfied by assumption. In this case, $\gcd(\delta_{c_1},\delta_{c_2})=0$, so (Sp2) simplifies to deciding whether there exists an element $\Delta\in\MZ$ such that $f_{c_1}=f_{c_2}^{\Delta}$ in $\base$.
Define $m_1,M_1,m_2,M_2\in\MZ$ as
$$
m_i=m(f_{c_i}), \ \ \ M_i=M(f_{c_i}).
$$
    Without loss of generality, assume that $m_1\le m_2$. Let $\Delta=m_2-m_1$. It can be determined whether $f_{c_1}=f_{c_2}^{\Delta}$ by scanning pairs of elements starting from $m_1$ to $\max(M_1,M_2)$. Clearly, this can be done in linear time.

        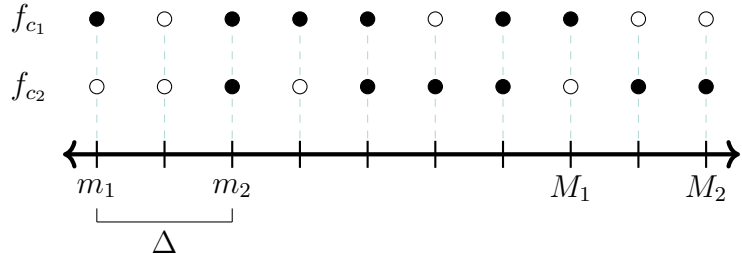
\begin{figure}[!h]
        \centering
        \begin{tikzpicture}[scale=0.9]
	\draw[help lines, dashed, ystep=0] (-5,0) grid(4,2);
	\draw[<->,ultra thick] (-5.5,0)--(4.5,0);
	\draw[-, thick] (-5,-0.2)--(-5,0.2);
	\draw[-, thick] (-4,-0.2)--(-4,0.2);
	\draw[-, thick] (-3,-0.2)--(-3,0.2);
	\draw[-, thick] (-2,-0.2)--(-2,0.2);
	\draw[-, thick] (-1,-0.2)--(-1,0.2);
	\draw[-, thick] (0,-0.2)--(0,0.2);
	\draw[-, thick] (4,-0.2)--(4,0.2);
	\draw[-, thick] (3,-0.2)--(3,0.2);
	\draw[-, thick] (2,-0.2)--(2,0.2);
	\draw[-, thick] (1,-0.2)--(1,0.2);
	\node[draw=none] at (-6,1) {$f_{c_2}$};
	\node[draw=none] at (-6,2) {$f_{c_1}$};

	\filldraw[black] (0,1) circle (3pt);
	\filldraw[black] (1,1) circle (3pt);
	\filldraw[black] (3,1) circle (3pt);
	\filldraw[black] (-1,1) circle (3pt);
	\filldraw[black] (-3,1) circle (3pt);
	\filldraw[black] (4,1) circle (3pt);
        \filldraw[draw=black,fill=white] (-5,1) circle (3pt);
        \filldraw[draw=black,fill=white] (-4,1) circle (3pt);
        \filldraw[draw=black,fill=white] (-2,1) circle (3pt);
        \filldraw[draw=black,fill=white] (2,1) circle (3pt);
	
	\filldraw[black] (-3,2) circle (3pt);
	\filldraw[black] (-2,2) circle (3pt);
	\filldraw[black] (-1,2) circle (3pt);
	\filldraw[black] (1,2) circle (3pt);
	\filldraw[black] (-5,2) circle (3pt);
	\filldraw[black] (2,2) circle (3pt);
        \filldraw[draw=black,fill=white] (-4,2) circle (3pt);
        \filldraw[draw=black,fill=white] (0,2) circle (3pt);
        \filldraw[draw=black,fill=white] (3,2) circle (3pt);
        \filldraw[draw=black,fill=white] (4,2) circle (3pt);

        \node[draw=none] at (-5,-0.5) {$m_1$};
        \node[draw=none] at (-3,-0.5) {$m_2$};
        \node[draw=none] at (2,-0.5) {$M_1$};
        \node[draw=none] at (4,-0.5) {$M_2$};

        \draw [-] (-5,-0.8) to (-5,-1);
        \draw [-] (-3,-0.8) to (-3,-1);
        \draw [-] (-5,-1) to (-3,-1);
        \node[draw=none] at (-4,-1.3) {$\Delta$};

	\end{tikzpicture}
        \caption{An example of an instance where $f_{c_1}=f_{c_2}^{\Delta}$.}\label{fig:caseI}
        \end{figure}

        \textsc{(Case-II)} The second case is where $\delta_{c_1}=\delta_{c_2}\ne 0$. Let $\delta=\delta_{c_1}$. Here the equation simplifies to $\rb{1-x^{-\delta}}\rb{f_{z_1}-f_{z_2}}=f_{c_2}^{\delta_{z_2}}-f_{c_1}^{\delta_{z_1}}$, which has a solution if and only if there exists $\Delta\in\MZ$ such that
        $\pi_\delta(f_{c_1})=\pi_\delta(f_{c_2})^{\Delta}$ in $\MZ_2^{\MZ_{\delta}}$. 
        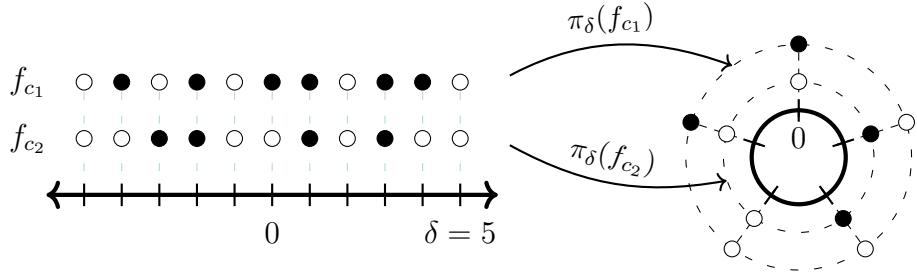
\begin{figure}[!h]
        \centering
        \begin{tikzpicture}[decoration={
  markings,
  mark=at position 0.2 with {\arrow{<}}}
]

	\draw[help lines, loosely dashed, ystep=0, xstep=0.5] (-5,0) grid(0,1.5);
	\draw[<->,ultra thick] (-5.5,0)--(0.5,0);
	\draw[-, thick] (-5,-0.15)--(-5,0.15);
        \draw[-, thick] (-4.5,-0.15)--(-4.5,0.15);
	\draw[-, thick] (-4,-0.15)--(-4,0.15);
        \draw[-, thick] (-3.5,-0.15)--(-3.5,0.15);
	\draw[-, thick] (-3,-0.15)--(-3,0.15);
        \draw[-, thick] (-2.5,-0.15)--(-2.5,0.15);
	\draw[-, thick] (-2,-0.15)--(-2,0.15);
        \draw[-, thick] (-1.5,-0.15)--(-1.5,0.15);
	\draw[-, thick] (-1,-0.15)--(-1,0.15);
        \draw[-, thick] (-0.5,-0.15)--(-0.5,0.15);
	\draw[-, thick] (0,-0.15)--(0,0.15);
	\node[draw=none] at (-5.75,0.75) {$f_{c_2}$};
	\node[draw=none] at (-5.75,1.5) {$f_{c_1}$};

	\filldraw[draw=black,fill=white] (-5,1.5) circle (3pt);
        \filldraw[draw=black] (-4.5,1.5) circle (3pt);
        \filldraw[draw=black,fill=white] (-4,1.5) circle (3pt);
        \filldraw[draw=black] (-3.5,1.5) circle (3pt);
        \filldraw[draw=black,fill=white] (-3,1.5) circle (3pt);
        \filldraw[draw=black] (-2.5,1.5) circle (3pt);
        \filldraw[draw=black] (-2,1.5) circle (3pt);
        \filldraw[draw=black,fill=white] (-1.5,1.5) circle (3pt);
        \filldraw[draw=black] (-1,1.5) circle (3pt);
        \filldraw[draw=black] (-0.5,1.5) circle (3pt);

        \filldraw[draw=black,fill=white] (-5,0.75) circle (3pt);
        \filldraw[draw=black,fill=white] (-4.5,0.75) circle (3pt);
        \filldraw[draw=black] (-4,0.75) circle (3pt);
        \filldraw[draw=black] (-3.5,0.75) circle (3pt);
        \filldraw[draw=black,fill=white] (-3,0.75) circle (3pt);
        \filldraw[draw=black,fill=white] (-2.5,0.75) circle (3pt);
        \filldraw[draw=black] (-2,0.75) circle (3pt);
        \filldraw[draw=black,fill=white] (-1.5,0.75) circle (3pt);
        \filldraw[draw=black] (-1,0.75) circle (3pt);
        \filldraw[draw=black,fill=white] (-0.5,0.75) circle (3pt);

		\node[circle, draw=black, minimum size =1.25cm, ultra thick] (c) at (4.5,0.5){};
        \node[circle, draw=black, minimum size=2cm, loosely dashed] (c1) at (4.5,0.5){};
        \node[circle, draw=black, minimum size=3cm, loosely dashed] (c2) at (4.5,0.5){};
		\node[draw=none] at ($(c.north)-(0,0.4)$) {$0$};
		
		\foreach \X in {90,162,...,405}
            {\draw[thick] (c.\X) -- ++ (\X:0.12);
            \draw[thick] (c.\X) -- ++ (180+\X:0.12);}

            \foreach \X in {90,162,...,405}
            {\draw[dashed] (c.\X) -- ++ (\X:1);
            \draw[thick] (c.\X) -- ++ (180+\X:0.18);}

            \draw[->, thick] (0.5,1.5) edge [out=30, in=160] node [midway, above, sloped] (p) {\small$\pi_\delta(f_{c_1})$}  ($(c2.126)$);
            \filldraw[draw=white,fill=white] (0.5,1.5) circle (5pt);
            \filldraw[draw=black,fill=white] (0,1.5) circle (3pt);

            \draw[->, thick] (0.5,0.75) edge [out=-30, in=190] node [midway, above, sloped] (p2) {\small$\pi_\delta(f_{c_2})$}  ($(c1.198)$);
            \filldraw[draw=white,fill=white] (0.5,0.75) circle (5pt);
            \filldraw[draw=black,fill=white] (0,0.75) circle (3pt);
        \filldraw[draw=black, fill=white] (c2.234) circle (3pt);
        \filldraw[draw=black,fill=white] (c2.306) circle (3pt);
        \filldraw[draw=black,fill=white] (c2.378) circle (3pt);
		
		\filldraw[black] (c2.north) circle (3pt);
		
		\filldraw[black] (c2.162) circle (3pt);

 \filldraw[draw=black,fill=white] (c1.north) circle (3pt);
		\filldraw[draw=black,fill=white] (c1.162) circle (3pt);
		\filldraw[draw=black,fill=white] (c1.234) circle (3pt);
		\filldraw[black] (c1.306) circle (3pt);
		\filldraw[black] (c1.378) circle (3pt);

 \node[draw=none] at (-2.5,-0.5) {$0$};
 \node[draw=none] at (0,-0.5) {$\delta=5$};

	\end{tikzpicture}
        \caption{An example of an instance where $\pi_\delta(f_{c_1})=\pi_\delta(f_{c_2})^{\Delta}$. Here, $\delta=5$ and $\Delta=-3$.}
        \end{figure}
        In other words, the equation has a solution if and only if the images of $f_{c_1}$ and $f_{c_2}$ are cyclic permutations of each other in the quotient space defined by the $\delta$-component of the functions. Let $\Sigma=\{0,1\}$. Denote by $s_1, s_2\in \Sigma^{\delta}$ the bitstring representations of the words $w_1$ and $w_2$ representing $\pi_\delta(f_{c_1})$ and $\pi_\delta(f_{c_2})$ respectively. We know that $s_1$ and $s_2$ are cyclic permutations of each other if and only if $s_1$ appears somewhere in $s_2^2$. This can be checked in linear time using the 
        Knuth--Morris--Pratt pattern-matching algorithm, see \cite{KMP,Matiyasevich:1973}.
\end{proof}

To explain the idea behind the proof of the next lemma
let us consider an element $(\delta,f)\in L_2$ with $\delta>0$.
Let $i\in \supp(f)$. Then
$$
(0,-x^i)^{-1}(\delta,f)(0,-x^i)=
(\delta,-(1-x^{-\delta})x^i+f)=
(\delta,f-x^i+x^{i-\delta}).
$$
Hence, an (elementary) conjugation by $(0,-x^i) = t^iat^{-i}\in Y$ 
switches off the lamp at position $i$
and changes the state of the lamp at position $i-\delta$.
One can think that such an operation ``pulls'' the state of the lamp 
at position $i$
$\delta$ positions left
onto the state of the lamp at position $i-\delta$.
In a similar fashion, conjugating $(\delta,f)$ 
by $(0,x^{i+\delta})$ pulls the state from the position $i$
to the position $i+\delta$.
Therefore, there exists a sequence of elementary conjugations that
\begin{itemize}
\item[(L)]
pulls the lamps lit by $f$ right of the position $\delta-1$ 
all the way \textbf{left} to positions $\{0,\dots,\delta-1\}$ (starting from the rightmost
lit lamp and proceeding left);
\item[(R)]
pulls the lamps lit by $f$ left of the position $0$ 
all the way \textbf{right} to positions $\{0,\dots,\delta-1\}$
(starting from the leftmost lit lamp and proceeding right).
\end{itemize}
Notice that a single operation of pulling a lamp $\delta$
units left/right can be done in constant time if
with the lamp at position $i$ we keep the pointers
to the lamps at positions $i\pm\delta$, i.e., 
if we prepare a linked structure similar to one shown in Figure \ref{fig:linked}.

        \begin{figure}[!h]
        \centering
        \begin{tikzpicture}
        \draw[-, thick] (-2,0)--(3,0);
        \draw[-, thick] (-2,1)--(3,1);

        \draw[-, thick] (0,0)--(0,1);
        \draw[-, thick] (1,0)--(1,1);
        \draw[-, thick] (2,0)--(2,1);
        \draw[-, thick] (-1,0)--(-1,1);
        \draw[-, thick] (-2,0)--(-2,1);
        \draw[-, thick] (3,0)--(3,1);

        \node[draw=none] at (4,0.5) {$\cdots$};
        \node[draw=none] at (-3,0.5) {$\cdots$};

        \draw[-, thick] (5.5,0)--(7,0);
        \draw[-, thick] (5.5,1)--(7,1);
        \draw[-, thick] (-4.5,0)--(-6,0);
        \draw[-, thick] (-4.5,1)--(-6,1);

        \draw[-, thick] (6,0)--(6,1);
        \draw[-, thick] (7,0)--(7,1);
        \draw[-, thick] (-5,0)--(-5,1);
        \draw[-, thick] (-6,0)--(-6,1);

        \node[draw=none] at (0.5,0.5) {\footnotesize{$f(1)$}};
        \node[draw=none] at (1.5,0.5) {\footnotesize{$f(2)$}};
        \node[draw=none] at (2.5,0.5) {\footnotesize{$f(3)$}};
        \node[draw=none] at (-0.5,0.5) {\footnotesize{$f(0)$}};
        \node[draw=none] at (-1.5,0.5) {\footnotesize{$f(-1)$}};
        \node[draw=none] at (-5.5,0.5) {\footnotesize{$f(m)$}};
        \node[draw=none] at (6.5,0.5) {\footnotesize{$f(M)$}};

        \draw[stealth-stealth, thick] (-1.5,0) to[out=-40,in=-140] (1.5,0);
        \draw[stealth-stealth, thick] (-0.5,0) to[out=-40,in=-140] (2.5,0);
        \draw[stealth-stealth, thick] (0.5,0) to[out=-40,in=-140] (3.5,0);
        \draw[stealth-stealth, thick] (1.5,0) to[out=-40,in=-140] (4.5,0);
        \draw[stealth-stealth, thick] (2.5,0) to[out=-40,in=-140] (5.5,0);
        \draw[stealth-stealth, thick] (-2.5,0) to[out=-40,in=-140] (0.5,0);
        \draw[stealth-stealth, thick] (-3.5,0) to[out=-40,in=-140] (-0.5,0);
        \draw[stealth-stealth, thick] (-4.5,0) to[out=-40,in=-140] (-1.5,0);
        \draw[stealth-stealth, thick] (3.5,0) to[out=-40,in=-140] (6.5,0);
        \draw[stealth-stealth, thick] (-5.5,0) to[out=-40,in=-140] (-2.5,0);

        \fill [white] (3.025,-0.75) rectangle (5.5,0.025);
        \fill [white] (-4.5,-0.75) rectangle (-2.025,0.025);
	
	\end{tikzpicture}
        \caption{Double-$\delta$-linked representation for $f$ with $\delta=3$.}\label{fig:linked}
        \end{figure}
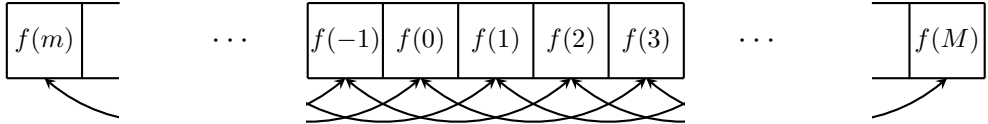


Such sequence transforms $(\delta,f)$ into $(\delta,f')$ with
$\supp(f') \subseteq\{0,\dots,\delta-1\}$
using a conjugator $c=rl$ that consists of two parts, $l$ and $r$, where $l$ pulls lamps to the left and $r$ pulls lamps to the right.
The word $l$ is a product of elementary conjugators 
$$
l=(t^iat^{-i})\cdot (t^jat^{-j})\cdot(t^kat^{-k})\dots=
t^iat^{j-i}at^{k-j}a\cdots
$$
where $i=M(f)>j>k>\dots>\delta-1$; its length is bounded by $3|M(f)|$.
The word $r$ is constructed in a similar fashion and 
satisfies a similar bound $|r|\le 3|m(f)|$.
It is straightforward to construct both words; one can do it in linear time.
Furthermore, since each elementary conjugation
shifts a lamp $\delta$ units to the left or to the right,
it follows that $\pi_\delta(f)=\pi_\delta(f')$ and 
the function $f'$ restricted to the set $\{0,\dots,\delta-1\}$ 
gives exactly $\pi_\delta(f)$ and, so, slightly abusing notation
we get
$$
c^{-1} (\delta,f) c = (\delta,\pi_\delta(f)),
$$
viewing $\pi_\delta(f)$ as a function from $\MZ$ to $\MZ_2$.

\begin{lemma}\label{le:linear_conj}
Let $w=w(a,t)$ be a word in the alphabet of $L_2$ 
defining an element $g=(\delta,f)$ with $\delta>0$.
Then there exists a conjugator $c=c(a,t)$ of length bounded by  $3|w|$ 
such that $c^{-1}gc=(\delta,f')$ with 
$\supp(f')\subseteq\{0,\dots,\delta-1\}$.
Furthermore, $c$ can be computed in linear time.
\end{lemma}


Recall that the \textbf{conjugacy search} problem for $L_2$ requires finding $x\in L_2$ satisfying $x^{-1} c_1 x=c_2$ for given conjugate elements $c_1,c_2\in L_2$.

\begin{theorem}
There is an algorithm that, for given conjugates $c_1,c_2\in L_2$, produces a conjugator $x=x(a,t)$ in linear time $O(|c_1|+|c_2|)$.
\end{theorem}

\begin{proof}
Consider two conjugate elements $c_1,c_2\in L_2$. 
Compute $(\delta_{c_1},f_{c_1})$ and $(\delta_{c_2},f_{c_2})$.

\textsc{(Case-I)}
If $\delta_{c_1}=\delta_{c_2}=0$, then, as in the proof of
Proposition \ref{pr:conjugacy-linear}, use the KMP algorithm
to find $\Delta\in\MZ$ satisfying $f_{c_1}^\Delta=f_{c_2}$ in linear time and output $x=t^{\Delta}$, as illustrated in Figure \ref{fig:caseI}.

\textsc{(Case-II)}
If $\delta_{c_1}=\delta_{c_2} \ne 0$, then
without loss of generality we may assume that 
$\delta_{c_1}=\delta_{c_2} > 0$ and, as in the proof of
Proposition \ref{pr:conjugacy-linear}, use the KMP algorithm
to find $\Delta\in\MZ$ where $0\le\Delta<\delta$ satisfying 
$$
\pi_{\delta}(f_{c_1}^\Delta) = \pi_{\delta}(f_{c_1})^\Delta=\pi_{\delta}(f_{c_2}).
$$
Then apply Lemma \ref{le:linear_conj} twice:
\begin{itemize}
\item 
find $x_1=x_1(a,t)$ such that
$x_1^{-1} (\delta,f_{c_1}^\Delta) x_1= (\delta,\pi_{\delta}(f_{c_1}^\Delta))$;
\item 
find $x_2=x_2(a,t)$ such that
$x_2^{-1} (\delta,f_{c_2}) x_2 = (\delta,\pi_{\delta}(f_{c_2}))$.
\end{itemize}
An immediate consequence is that 
$x_1^{-1} t^{-\Delta} c_1 t^{\Delta} x_1 = x_2^{-1} c_2 x_2$
and the algorithm can output $x=t^{\Delta} x_1 x_2^{-1}$.
\end{proof}

\section{Orientable equations}
\label{se:orientable}

Here we study quadratic equations defined by \eqref{eq:orientable} 
over $\lamp$.
Recall that the \emph{derived subgroup} $[G,G]$ of a group $G$
is defined as
$$
[G,G]=\gp{[x,y] \mid x,y\in G}.
$$
The \emph{commutator width} of $G$ is the least $k\in\MN \cup\{\infty\}$ such that every $g\in [G,G]$ can be expressed
as a product of at most $k$ commutators.

\begin{lemma}\label{le:commutator-width}
The following hold:
\begin{itemize}
\item[(a)]
$[\lamp,\lamp] = \ker(\sigma_a) \cap \ker(\sigma_t)$.
\item[(b)]
The commutator width of $\lamp$ is one.
\end{itemize}
\end{lemma}

\begin{proof}
It is easy to check that every
product of commutators belongs to $\ker(\sigma_a) \cap \ker(\sigma_t)$
and, hence, $[\lamp,\lamp] \subseteq \ker(\sigma_a) \cap \ker(\sigma_t)$.
Conversely, if $(\delta,f) \in \ker(\sigma_a) \cap \ker(\sigma_t)$, then
$\delta=0$ and 
$\sum_{i=-\infty}^\infty f(i)=0$.
Define $g\colon\MZ\to\MZ_2$ by
$$
g(j)=\sum_{i\le j} f(i)
\ \mbox{ for } \ j\in\MZ.
$$
Clearly, $g$ satisfies $f(j)=g(j)-g(j-1)$ for every $j$
and has finite support
because $\sum_{i=-\infty}^\infty f(i)=0$.
Hence,
$$
(0,f)=
(0,g)(1,0)(0,g)^{-1}(1,0)^{-1},
$$
i.e., $(0,f)$ is a commutator.
Thus, $[\lamp,\lamp] \supseteq \ker(\sigma_a) \cap \ker(\sigma_t)$.

Finally, notice that the above argument shows that every element 
in $[\lamp,\lamp]$ is a single commutator, which establishes (b).
\end{proof}

\begin{proposition} \label{pr:oqe_p}
The Diophantine problem for orientable quadratic equations over $\lamp$ 
is decidable in linear time.
\end{proposition}

\begin{proof}
Indeed, for an orientable equation \eqref{eq:orientable} of genus $g\ge 1$ we have
\begin{align*}
\mbox{\eqref{eq:orientable} has a solution}
&\ \ \stackrel{\ref{le:commutator-width}(b)}{\Leftrightarrow}\ \ 
\prod_{j=1}^k z_j^{-1}c_jz_j \in [\lamp,\lamp]
\mbox{ for some } z_1,\dots,z_k\\
&\ \ \stackrel{\ref{le:commutator-width}(a)}{\Leftrightarrow}\ \ 
\sigma_a\rb{\prod_{j=1}^k z_j^{-1}c_jz_j}=0
\ \wedge\ 
\sigma_t\rb{\prod_{j=1}^k z_j^{-1}c_jz_j}=0 \\
&\ \ \stackrel{\phantom{\ref{le:commutator-width}(a)}}{\Leftrightarrow}\ \ 
\sigma_a(c_1\cdots c_k)=0
\ \wedge\ 
\sigma_t(c_1\cdots c_k)=0,
\end{align*}
where the last equivalence can be seen from the fact that $\sigma_a$ and $\sigma_t$
are homomorphisms into abelian groups in which $z_j^{-1}$ and $z_j$ cancel out.
Finally, the condition
$\sigma_a(c_1\cdots c_k)=0 \ \wedge\ \sigma_t(c_1\cdots c_k)=0$
can be checked in linear time.
\end{proof}

\section{Non-orientable equations}
\label{se:nonorientable}

In this section we analyze non-orientable equations in the lamplighter group $L_2$.
We say that $g\in G$ is a \emph{square} if $g=h^2$ for some $h\in G$.

\begin{lemma}\label{le:square}
$(\delta,f)\in L_2$ is a square if and only if 
$\delta$ is even and $\pi_{\delta/2}(f)=0$.
\end{lemma}

\begin{proof}
``$\Leftarrow$''
If $\delta=2b$ is even and $\pi_b(f)=0$, then we have
\begin{align*}
\pi_b(f)=0
&\ \Rightarrow\ 
f\in\ker(\pi_b) = \im(\lambda_b)
&&\mbox{(Corollary \ref{cor:sequencemulti})}\\
&\ \Rightarrow\ 
f = (1-x^b)\cdot g &&\mbox{(for some $g\in\base$)}\\
&\ \Rightarrow\ 
f = (1+x^b)\cdot g = (x^{-b}+1)\cdot (x^b\cdot g) 
&&\mbox{(characteristic $2$).}
\end{align*}
Hence, $(2b,f)=(b,x^b\cdot g)^2$ is a square.

``$\Rightarrow$''
Suppose that $(\delta,f)=(b,g)^2$ is a square. 
Then $(\delta,f) = (2b,(x^{-b}+1)g)) = (2b,g^b+g))$.
Then $\delta=2b$ and $f=g+g^b$.
Clearly, $\pi_b(g^b)=\pi_b(g)$ and, hence, 
$\pi_b(g^b+g)=0$.
Hence, the result.
\end{proof}

Let $V$ be the verbal subgroup of $L_2$ generated by all squares, i.e.,
$V=\gpr{g^2}{g\in L_2}$. 
It is easy to see that $V$ is a normal subgroup of $L_2$ containing 
the derived subgroup $[L_2,L_2]$, because every commutator can be expressed
as a product of squares as follows:
$$
[g, h]=g^{-1} h^{-1} g h=g^{-2} (g h^{-1})^2 h^{2}
$$
for every $g,h$.
In particular, $L_2/V$ is abelian.

\begin{lemma}\label{le:V-membership}
$L_2/V\simeq \MZ_2\times\MZ_2$.
\end{lemma}

\begin{proof}
Clearly, $\sigma_t(g)$ is even and $\sigma_a(g)=0$  
for every square $g=h^2\in L_2$;
this also applies to any product of squares.
Using the contrapositive form of that observation,
the elements $a,t,at$ do not belong to $V$. Similarly we verify that
the elements $1,a,t,at$ are distinct in $L_2/V$.
The quotient group $L_2/V$ is an abelian group in which every nontrivial element
has order $2$. Since $L_2$ is two-generated, its abelian quotient
$L_2/V$ cannot contain more than four elements. Therefore, 
$L_2/V\simeq \MZ_2\times\MZ_2$.
\end{proof}

Lemma \ref{le:V-membership} gives a linear time algorithm
to decide if a word $w=w(a,t)$ defines an element of $V$
($w\in V$ if and only if $w=1$ in $L_2/V$).
By definition, the \emph{verbal width} of $V$ is
the least $k\in\MN\cup\{\infty\}$ such that if $g\in V$,
then $g$ is a product of at most $k$ squares.

\begin{proposition}\label{pr:V-verbal-width}
The verbal width of the subgroup $V$ is $2$.
\end{proposition}

\begin{proof}
By Lemma \ref{le:V-membership}, $a t a t^3\in V$
(its image is trivial in $L_2/V$)
and, by Lemma \ref{le:square}, $a t a t^3$ is not a square.
Hence, the verbal width of $V$ is greater than $1$.

Now, suppose that $(\delta,f)\in V$.
Then $\delta=2(1+k)$ for some $k\in\MZ$. Define
$g\colon\MZ\to \MZ_2$ by
$$
g(j)=\sum_{i=-\infty}^j f(i)
\ \ \mbox{ for } x\in\MZ.
$$
It is easy to check that
$$
(\delta,f) = (1,g)^2+(k,0)^2.
$$
Hence, every $(\delta,f)\in V$ can be expressed as a product of two squares
and the verbal width of $V$ is $2$.
\end{proof}

Proposition \ref{pr:V-verbal-width} implies
that there are exactly two cases for non-orientable equations
over the lamplighter group: genus one and genus two
(same as genus more than one).
The latter case is easy because $V$ is normal.
Consider a non-orientable equation of genus two
\begin{equation}\label{eq:non-orientable-g2}
x^2y^2\prod z_i^{-1} c_i z_i=1.
\end{equation}

\begin{proposition}
The Diophantine problem for equations 
\eqref{eq:non-orientable-g2}
is decidable in linear time.
\end{proposition}

\begin{proof}
By Proposition \ref{pr:V-verbal-width},
\eqref{eq:non-orientable-g2} has a solution
if and only if
$\prod z_i^{-1} c_i z_i \in V$
for some $z_1,\dots,z_k\in L_2$, which is true if and only if
$\sigma_t(c_1\cdots c_k)$ is even and $\sigma_a(c_1\cdots c_k)=0$.
The obtained condition can be checked in linear time.
\end{proof}

Consider a non-orientable equation of genus one
\begin{equation}\label{eq:non-orientable-g1}
x^2\prod z_i^{-1} c_i z_i=1.
\end{equation}
Clearly, if the equation has a solution, then 
$2\sigma_t(x)=-\sum \sigma_t(c_i)$, i.e., $\sigma_t(x)$
is uniquely defined by the equation.

\begin{lemma}\label{le:dropping-genus}
If $c_1\cdots c_k=(0,f)$, then
$$
\underbrace{x^2\prod z_i^{-1} c_i z_i=1}_W \mbox{ has a solution }
\ \Leftrightarrow\ 
\underbrace{\prod z_i^{-1} c_i z_i=1}_{W'} \mbox{ has a solution.}
$$
\end{lemma}

\begin{proof}
Name the genus one equation $W$ and the spherical equation $W'$.
If $W'$ has a solution, then the same assignment for $z_1,\dots,z_k$
and $x=1$ gives a solution for $W$.
Conversely, if $W$ has a solution with $x=(\delta,f)$, then
$\delta=0$.
Notice that $(0,f)^2 = 1$ in $L_2$
for every $f\colon\MZ\to\MZ_2$. Hence, we may assume that $f=0$
and, $W'$ has a solution.
\end{proof}

\begin{proposition}
The Diophantine problem 
for equations \eqref{eq:non-orientable-g1} is $\NP$-hard.
\end{proposition}

\begin{proof}
For an instance $W':\prod_{i=1}^{k} z_i^{-1}c_{i}z_i=1$ of $\DP_{SPH_0}$, construct an equation $W: x^2\prod_{i=1}^{k} z_i^{-1}c_{i}z_i=1$ of type 
\eqref{eq:non-orientable-g1}.
By definition of $\DP_{SPH_0}$,  $c_1,\dots,c_k\in Y$ and, hence,
the assumption of Lemma \ref{le:dropping-genus} holds,
which implies that $W$ has a solution if and only if $W'$ has a solution.
This establishes a polynomial-time reduction from $\DP_{SPH_0}$,
which is $\NP$-hard by Corollary \ref{co:SPH-0-hard},
to the Diophantine problem for equations \eqref{eq:non-orientable-g1}.
\end{proof}

\begin{theorem}\label{th:genus-one-nonorientable-NP}
The Diophantine problem for equations \eqref{eq:non-orientable-g1} 
belongs to $\NP$.
\end{theorem}

\begin{proof}
Consider an equation $W=1$ of type \eqref{eq:non-orientable-g1}
with constants $c_i=(\delta_{c_i},f_{c_i})$,
that has a solution 
$x=(\delta_x,f_x),z_1=(\delta_1,f_1),\dots,z_k=(\delta_k,f_k)$. 
First, notice that
$$
\delta_x = \sigma_t(x) = 
-\tfrac{1}{2}(\delta_{c_1}+\dots+\delta_{c_k})\in\MZ.
$$
If $\delta_x=0$, then, by Lemma \ref{le:dropping-genus},
the equation has a solution if and only if the spherical
equation obtained by dropping $x^2$ has a solution.
By Theorem \ref{th:spherical-NP}, spherical equations have $\NP$-certificates.

If $\delta_x\ne 0$, then $\delta_{c_i}\ne 0$ for some $i$, and
$$
0<\gcd(\delta_x,\delta_{c_1},\dots,\delta_{c_k}) = d \le |W|.
$$
The second component of the equation $W=1$ (as an element of $\MZ\ltimes \MZ_2^\MZ$) is of the form
\begin{equation}\label{eq:component2-a}
(1+x^{-\delta_x})x^{-\Delta_0}f_x+
\sum_{i=1}^k 
x^{-\Delta_{i}}
\left[
(1-x^{-\delta_{c_i}})f_{i}+
x^{-\delta_{i}}f_{c_i}
\right]
=0,
\end{equation}
where $\Delta_i=\delta_{c_{i+1}}+\dots+\delta_{c_k}$ for $i=0,\dots,k$,
and variables 
$f_x,f_1,\dots,f_k\in\MZ_2^\MZ$ and 
$\delta_1,\dots,\delta_k\in\MZ$. Notice that $$(1+x^{-\delta_x})x^{-\Delta_0}f_x=(1-x^{-\delta_x})x^{-\Delta_0}f_x$$ because $\MZ_2$ has characteristic $2$, so \eqref{eq:component2-a} becomes
\begin{equation}\label{eq:component2-new}
\left((1-x^{-\delta_x})x^{-\Delta_0}f_x+\sum_{i=1}^k 
(1-x^{-\delta_{c_i}})x^{-\Delta_{i}}f_{i}\right)+
\sum_{i=1}^k
x^{-\delta_{i}-\Delta_i}f_{c_i}
=0.
\end{equation}
By Proposition \ref{pr:transform-components},
\eqref{eq:component2-new} has a solution if and only if
the following equation:
\begin{equation}\label{eq:component2-b}
x^{\delta_1'}f_{c_1}+\dots+x^{\delta_k'}f_{c_k}=0
\end{equation}
has a solution in $\MZ_2^{\MZ_d}$ with $\delta_1',\dots,\delta_k' \in\MZ_d$. A solution of \eqref{eq:component2-b} establishes an $\NP$-certificate
in this case.
\end{proof}

\begin{corollary}\label{co:equation3-complexity}
There is an algorithm that for a given non-orientable
equation \eqref{eq:non-orientable-g1}, decides 
if it has a solution or not in time $O(|W|^k)$.
\end{corollary}

\begin{proof}
We can reuse parts of the proof 
of Theorem \ref{th:genus-one-nonorientable-NP}.
If $\delta_x=0$, then the problem reduces to 
a spherical equation that,
by Corollary \ref{co:equation1-complexity},
can be solved in time $O(|W|^k)$.

If $\delta_x\ne 0$, then
the problem reduces to the equation \eqref{eq:component2-b} 
with unknowns $\delta_1,\dots,\delta_{k}\in\MZ_d$,
which can be solved
as in the second part 
of the proof of Corollary \ref{co:equation1-complexity}
(the case when $\delta\ne 0$),
by enumerating all tuples $\delta_1,\dots,\delta_{k-1}\in\MZ_d$
and testing if the tuple$(\delta_1',\dots,\delta_{k-1}',0)$ satisfies
\eqref{eq:component2-b}. This can be done in $O(|W|^k)$ time.
\end{proof}

\section{Parametric complexity}

In this section, we summarize some of our previous results
and state that the Diophantine problem for quadratic equations
(of all three types \eqref{eq:spherical}, \eqref{eq:orientable}, 
and \eqref{eq:nonorientable}) over $L_2$
can be solved in polynomial time for any fixed bound on the number 
of variables $k$, i.e., the Diophantine 
problem belongs to the complexity class $\XP$ 
(``slicewise'' polynomial time).

\begin{corollary}
There is an algorithm that solves the Diophantine problem 
for quadratic equations $W=1$ of type
\eqref{eq:spherical},
\eqref{eq:orientable},
\eqref{eq:nonorientable}
over the lamplighter group $L_2$ in time $O(|W|^k)$.
\end{corollary}

\begin{proof}
Follows from 
Corollary \ref{co:equation1-complexity},
Proposition \ref{pr:oqe_p},
and 
Corollary \ref{co:equation3-complexity}.
\end{proof}

\begin{corollary}
Fix $K\in \MN$. The Diophantine problem for the class of 
quadratic equations with $k\le K$ is decidable in polynomial time.
The Diophantine problem for quadratic equations over the lamplighter
group $L_2$ belongs to $\XP$.
\end{corollary}

\bibliography{main_bibliography}

\end{document}